\documentclass[12pt]{amsart}

\usepackage{amsmath,amssymb,mathtools}
\usepackage{microtype}

\newtheorem{thm}{Theorem}[section]
\newtheorem{prop}[thm]{Proposition}
\newtheorem{cor}[thm]{Corollary}
\newtheorem{lem}[thm]{Lemma}
\theoremstyle{definition}
\newtheorem{defn}[thm]{Definition}
\theoremstyle{remark}
\newtheorem{remark}[thm]{Remark}
\newtheorem{example}[thm]{Example}

\numberwithin{equation}{section}

\newcommand{\Hh}{\mathbb H}
\newcommand{\C}{\mathbb C}
\newcommand{\Z}{\mathbb Z}

\begin{document}

\title[Replication descent and $J$-finality]{Replication Descent and $J$-Finality of Replicable Functions}
\author{Eric Culf}
\author{Abdellah Sebbar}

\address{Institute for Quantum Computing, University of Waterloo}
\email{eculf@uwaterloo.ca}
\address{Department of Mathematics and Statistics, University of Ottawa}
\email{asebbar@uottawa.ca}

\subjclass[2020]{11F03, 11F22, 11F30}
\keywords{replicable function, completely replicable function, Hauptmodul, Faber polynomial, modular equation, moonshine, $J$-function}

\begin{abstract}
We show that replication carries congruence symmetry down an explicit level
tower.  If a normalized replicable function $f$, holomorphic on the upper
half-plane, is invariant under $\Gamma_0(N)$, then its $n$th replicate is
invariant under $\Gamma_0(N/(N,n))$.  Thus every replicate whose index is
divisible by $N$ is the normalized modular invariant $J=j-744$.  This gives a
direct and classification-free proof of $J$-finality for congruence-invariant
replicable functions.  The argument uses only the replication identities and
an elementary generation theorem for congruence subgroups; it requires neither
complete replicability nor arithmetic hypotheses on the Fourier coefficients.
We then apply the descent law to completely replicable functions of finite
replication order.  Their replication towers have a canonical terminal
replicate, and the only possible terminal functions are
$J$, $q^{-1}$, and $q^{-1}+q$.  Moreover, a finite-order completely replicable
function is $J$-final precisely when its terminal replicate is $J$, and any
symmetry beyond translations forces this alternative.
\end{abstract}

\maketitle

\section{Introduction}

Replicability arose in Norton's study of the McKay--Thompson series of
monstrous moonshine.  It packages a family of identities in which Faber
polynomials are expressed through generalized Hecke sums of auxiliary
functions, the \emph{replicates}.  The model example is the normalized modular
invariant
\[
J(\tau)=j(\tau)-744=q^{-1}+196884q+21493760q^2+\cdots,
\qquad q=e^{2\pi i\tau},
\]
whose replicates are all equal to $J$.  Conway and Norton found the same
structure throughout the moonshine family: replication by $a$ reflects the
power map $g\mapsto g^a$ on the underlying Monster conjugacy classes
\cite{conway-norton,norton}.

A central theme in the subject is the relation between the formal replication
identities and modular symmetry.  Cummins and Norton proved replicability for
rational Hauptmoduln attached to a broad class of genus-zero groups
\cite{cummins-norton}.  In the opposite direction, Martin, Kozlov, and Cummins
and Gannon obtained modularity and genus-zero results from complete
replicability or modular equations under additional hypotheses
\cite{martin,kozlov,cummins-gannon}.  Carnahan later placed related phenomena
in the geometry of Hecke-monic functions on moduli spaces of elliptic curves
with torsors \cite{carnahan}.

The question addressed here is different and more structural.  Suppose that a
replicable function is already known to have congruence symmetry.  How does
that symmetry change as one moves through its replicate tower?  Our main
result gives a uniform answer for every replication index.

\begin{thm}\label{thm:intro-descent}
Let
\[
f(\tau)=q^{-1}+\sum_{r\geq1}a_rq^r
\]
be a replicable function holomorphic on $\Hh$, and suppose that $f$ is invariant under $\Gamma_0(N)$. Then for every $n\geq1$ the replicate $f^{(n)}$ is invariant under
\[
\Gamma_0\!\left(\frac{N}{(N,n)}\right).
\]
\end{thm}

Prime by prime, replication by $p^r$ removes
$\min\{r,v_p(N)\}$ powers of $p$ from the guaranteed level.  Indices coprime
to $N$ preserve the level, whereas every index divisible by $N$ reaches level
one.  At that point holomorphy and the normalized principal part force the
replicate to be $J$.

\begin{cor}\label{cor:intro-J}
Under the hypotheses of Theorem~\ref{thm:intro-descent},
\[
f^{(n)}=J
\qquad\text{whenever }N\mid n.
\]
In particular, $f$ is $J$-final.
\end{cor}

This conclusion uses no genus-zero hypothesis.  In particular, every
normalized replicable Hauptmodul for a group containing some $\Gamma_0(N)$ is
$J$-final.  More generally, the proof requires no arithmetic condition on the
Fourier coefficients and no compatibility among iterated replicates.  Its main
input is a congruence generation lemma that converts the invariance visible in
the replication formula into the full group
$\Gamma_0(N/(N,n))$.  In this way the level formula is obtained directly from
the identities themselves.

For comparison, Carnahan obtained the same formula for integral completely
replicable functions by using the classification of Alexander, Cummins, McKay,
and Simons together with Ferenbaugh's computation of fixing groups
\cite{alexander-cummins-mckay-simons,ferenbaugh,carnahan}.  Norton's earlier
work treats indices prime to the level in the Hauptmodul setting, including
extensions beyond rational coefficients and genus zero
\cite{norton-nonmonstrous,cummins-norton}.  Theorem~\ref{thm:intro-descent}
handles arbitrary indices without invoking a classified list.

The descent theorem also clarifies the end of a finite replication tower.  If
$f$ is completely replicable of finite replication order $K$, then $f^{(K)}$
is independent of the chosen order, has replication order one, and is reached
from every branch after further replication.  We call it the \emph{terminal
replicate}.  Combining Kozlov's order-one modular-equation theorem with the
classification recorded by Gannon gives a short list of possibilities
\cite{kozlov,gannon-hauptmodul}.

\begin{thm}\label{thm:intro-terminal}
Let $f$ be a completely replicable function of finite order $K$. Then
\[
f^{(K)}\in\{J,\ q^{-1},\ q^{-1}+q\}.
\]
If the symmetry group of $f$ is larger than the translation group, then $f^{(K)}=J$.
\end{thm}

The exclusion of $q^{-1}-q$ from the order-one list is elementary: its second
replicate is $q^{-1}+q$.  The trichotomy therefore holds without integrality or
algebraicity assumptions on the coefficients.  The $J$ alternative has a
particularly simple intrinsic meaning: $f$ is $J$-final if and only if its
terminal replicate is $J$.  Finally, the nontranslation theorem of Cummins and
Gannon, followed by our congruence descent, shows that either of the two
trigonometric terminal functions can occur only when the original symmetry
group is precisely the translation group.  Thus congruence symmetry not only
descends through replication; in every finite tower with genuine modular
symmetry, it determines the terminal object.

\section{Replicability and finite replication order}

Throughout, normalized series are written in the form
\begin{equation}\label{eq:normalized-series}
f(q)=q^{-1}+\sum_{r\geq1}a_rq^r.
\end{equation}
In particular, the constant term is zero. For $n\geq1$, the $n$th Faber polynomial $\Phi_n(X)$ is the unique monic polynomial of degree $n$ such that
\begin{equation}\label{eq:Faber-normalization}
\Phi_n(f(q))=q^{-n}+O(q).
\end{equation}

\begin{defn}\label{def:replicable}
A normalized formal series $f$ is replicable if there are normalized formal series
\[
f^{(a)}(q)=q^{-1}+O(q),\qquad a\geq1,
\]
with $f^{(1)}=f$, such that for every $n\geq1$,
\begin{equation}\label{eq:replication}
\Phi_n(f(\tau))
=
\sum_{ad=n}\ \sum_{0\leq b<d}
 f^{(a)}\!\left(\frac{a\tau+b}{d}\right).
\end{equation}
The series $f^{(a)}$ is the $a$th replicate of $f$. A replicable function is a replicable series whose $q$-expansion converges on the punctured unit disc and hence defines a holomorphic function on $\Hh$.
\end{defn}

The replicate family is unique when it exists. The term with $a=n$ and $d=1$ in \eqref{eq:replication} is $f^{(n)}(n\tau)$, while all other replicates have smaller indices, giving an induction on $n$.

\begin{prop}\label{prop:analytic-replicates}
Let $f$ be a replicable function. Then every replicate $f^{(n)}$ is holomorphic on $\Hh$.
\end{prop}

\begin{proof}
Proceed by induction on $n$, simultaneously proving that the formal $q$-expansion of each replicate converges near infinity and extends holomorphically to $\Hh$. The assertion is true for $n=1$. Assume it for all replicates with index smaller than $n$. Define
\begin{equation}\label{eq:isolate-replicate}
G_n(\tau)
:=
\Phi_n(f(\tau))
-
\sum_{\substack{ad=n\\a<n}}
\ \sum_{0\leq b<d}
 f^{(a)}\!\left(\frac{a\tau+b}{d}\right).
\end{equation}
Every term on the right is holomorphic for $\tau\in\Hh$: the first because $f$ is holomorphic, and the remaining terms by the induction hypothesis. Thus $G_n$ is holomorphic on $\Hh$.

For $\operatorname{Im}(\tau)$ sufficiently large, the convergent Fourier expansions of the terms on the right may be substituted into \eqref{eq:isolate-replicate}. By the formal replication identity, the resulting Fourier expansion is precisely the formal series $f^{(n)}(n\tau)$. Therefore
\[
F_n(z):=G_n(z/n),\qquad z\in\Hh,
\]
is holomorphic on $\Hh$ and has near infinity exactly the formal $q$-expansion of $f^{(n)}$. Thus that formal series converges near infinity and its sum extends holomorphically to $\Hh$. This completes the induction.
\end{proof}

Carnahan records the same observation in his treatment of finite-order replicable functions \cite[Section~5]{carnahan}. The proof is recalled because the same isolation of the last replicate is used in the congruence descent argument.

\begin{defn}\label{def:complete}
A replicable function is completely replicable if every replicate is replicable and
\begin{equation}\label{eq:complete-composition}
\bigl(f^{(a)}\bigr)^{(b)}=f^{(ab)}
\end{equation}
for all positive integers $a,b$.
\end{defn}

Following Kozlov, a completely replicable function has finite replication order $K$ if
\begin{equation}\label{eq:finite-order}
f^{(s)}=f^{((s,K))}
\qquad(s\geq1).
\end{equation}
Here $K$ is a chosen integer satisfying the relation; minimality is not part of Kozlov's definition. The integer $K$ need not be introduced through modular invariance. It is an intrinsic periodicity condition on the replicate tower. Kozlov proved that if $f$ has order $K$, then $f^{(s)}$ has order $K/(s,K)$ in the corresponding sense \cite{kozlov}.

The normalized modular invariant is
\begin{equation}\label{eq:J}
J(\tau)=j(\tau)-744=q^{-1}+196884q+21493760q^2+\cdots.
\end{equation}
All replicates of $J$ are equal to $J$.

\begin{defn}\label{def:J-final}
A replicable function $f$ is $J$-final if
\[
f^{(n)}=J
\]
for at least one positive integer $n$.
\end{defn}

\section{A congruence generation lemma}

For $m,n\geq1$, write
\[
\Gamma^0(n)=
\left\{
\begin{pmatrix}a&b\\c&d\end{pmatrix}\in\mathrm{SL}_2(\Z):
 b\equiv0\pmod n
\right\}
\]
and
\[
\Gamma_0(m,n)=\Gamma_0(m)\cap\Gamma^0(n).
\]
Let
\[
T=\begin{pmatrix}1&1\\0&1\end{pmatrix}.
\]

\begin{lem}\label{lem:generation}
For all positive integers $m,n$,
\begin{equation}\label{eq:generation}
\Gamma_0(m)=\left\langle\Gamma_0(m,n),T\right\rangle.
\end{equation}
\end{lem}

\begin{proof}
Let
\[
\gamma=\begin{pmatrix}a&b\\c&d\end{pmatrix}\in\Gamma_0(m).
\]
First modify the upper-left entry without changing the upper-right entry. For $t\in\Z$ put
\[
L_m(t)=\begin{pmatrix}1&0\\mt&1\end{pmatrix}.
\]
Then $L_m(t)\in\Gamma_0(m,n)$ and
\[
\gamma L_m(t)=
\begin{pmatrix}
a+bmt&b\\c+dmt&d
\end{pmatrix}.
\]
The integer $t$ may be chosen so that
\begin{equation}\label{eq:a-prime-n}
(a+bmt,n)=1.
\end{equation}
Fix a prime $p\mid n$. If $p\mid m$ and $p\mid a$, then $p\mid c$ because $m\mid c$, and the determinant relation $ad-bc=1$ would be impossible modulo $p$. Hence, whenever $p\mid a$, one has $p\nmid m$. Since $(a,b)=1$, one also has $p\nmid b$. Thus, if $p\mid a$, there is exactly one residue class of $t$ modulo $p$ for which $a+bmt\equiv0\pmod p$, and it can be avoided. If $p\nmid a$, there is again at most one forbidden residue when $p\nmid bm$, while there is no forbidden residue when $p\mid bm$. Choosing an allowed residue modulo every prime divisor of $n$ and applying the Chinese remainder theorem gives \eqref{eq:a-prime-n}.

Set
\[
\gamma_1=\gamma L_m(t)=
\begin{pmatrix}a_1&b\\c_1&d\end{pmatrix},
\qquad (a_1,n)=1.
\]
Choose $k\in\Z$ such that
\[
a_1k+b\equiv0\pmod n.
\]
Then
\[
\gamma_1T^k=
\begin{pmatrix}
a_1&a_1k+b\\c_1&c_1k+d
\end{pmatrix}
\in\Gamma_0(m,n).
\]
Therefore
\[
\gamma=(\gamma_1T^k)T^{-k}L_m(t)^{-1}
\]
belongs to the subgroup generated by $\Gamma_0(m,n)$ and $T$. The reverse inclusion is immediate.
\end{proof}

The lemma converts invariance under the conjugated subgroup into invariance under the full congruence subgroup, without any coprimality assumption on $m$ and $n$.

\section{Congruence descent along the replicate tower}

For $a,d\geq1$ and $0\leq b<d$, write
\[
\alpha_{a,b,d}=\begin{pmatrix}a&b\\0&d\end{pmatrix}.
\]
Its determinant is $ad$ and it acts on $\Hh$ by
\[
\alpha_{a,b,d}\tau=\frac{a\tau+b}{d}.
\]

\begin{thm}\label{thm:congruence-descent}
Let $f$ be a normalized replicable function. Suppose that $f$ is invariant under $\Gamma_0(N)$. Then for every $n\geq1$,
\begin{equation}\label{eq:descent-level}
f^{(n)}\ \text{is invariant under}\ 
\Gamma_0\!\left(\frac{N}{(N,n)}\right).
\end{equation}
\end{thm}

\begin{proof}
The proof is by induction on $n$. For $n=1$ the assertion is the hypothesis. Assume it holds for every proper divisor of $n$ and set
\[
L=\operatorname{lcm}(N,n),
\qquad
M_a=\frac{N}{(N,a)}
\quad(a\mid n).
\]
One has the divisibility
\begin{equation}\label{eq:lcm-divisibility}
M_a=\frac{N}{(N,a)}\mid\frac{L}{a}.
\end{equation}
Indeed,
\[
aM_a=\operatorname{lcm}(a,N)
\]
divides $\operatorname{lcm}(n,N)=L$ because $a\mid n$.

Fix a factorization $ad=n$ with $a<n$ and define
\begin{equation}\label{eq:branch-sum}
S_{a,d}(\tau)=
\sum_{b\bmod d}
 f^{(a)}\!\left(\frac{a\tau+b}{d}\right).
\end{equation}
By the induction hypothesis, $f^{(a)}$ is invariant under $\Gamma_0(M_a)$. It remains to show that $S_{a,d}$ is invariant under $\Gamma_0(L)$.

Take
\[
\gamma=\begin{pmatrix}A&B\\C&D\end{pmatrix}\in\Gamma_0(L).
\]
Since $d\mid n\mid L$, reduction of $AD-BC=1$ modulo $d$ gives $AD\equiv1\pmod d$, so $A$ is invertible modulo $d$. For every $b\bmod d$ there is therefore a unique $b'\bmod d$ satisfying
\begin{equation}\label{eq:bprime}
Ab'\equiv aB+bD\pmod d.
\end{equation}
Consider
\[
\delta_{b}=\alpha_{a,b,d}\,\gamma\,\alpha_{a,b',d}^{-1}.
\]
A direct multiplication gives
\begin{equation}\label{eq:delta-matrix}
\delta_b=
\begin{pmatrix}
A+\dfrac{bC}{a}
&
\dfrac{aB+bD-Ab'}{d}-\dfrac{bb'C}{ad}
\\[3mm]
\dfrac{dC}{a}
&
D-\dfrac{b'C}{a}
\end{pmatrix}.
\end{equation}
All four entries are integers. For the diagonal entries this follows from $a\mid n\mid L\mid C$. The first term in the upper-right entry is integral by \eqref{eq:bprime}, and the second is integral because $ad=n\mid C$. The determinant is one because $\delta_b$ is a conjugated product of determinant-one.

It remains to check the lower-left congruence. By \eqref{eq:lcm-divisibility}, $M_a\mid L/a$, while $L\mid C$. Hence
\[
M_a\mid\frac{C}{a}\mid\frac{dC}{a}.
\]
Thus $\delta_b\in\Gamma_0(M_a)$. Consequently
\[
f^{(a)}\bigl(\alpha_{a,b,d}\gamma\tau\bigr)
=
f^{(a)}\bigl(\delta_b\alpha_{a,b',d}\tau\bigr)
=
f^{(a)}\bigl(\alpha_{a,b',d}\tau\bigr).
\]
The map $b\mapsto b'$ is a permutation modulo $d$: by \eqref{eq:bprime} it is an affine map with linear coefficient $A^{-1}D$, a unit modulo $d$. Thus summing over $b$ proves
\[
S_{a,d}(\gamma\tau)=S_{a,d}(\tau).
\]
Return to the replication identity and isolate the term $a=n,d=1$:
\begin{equation}\label{eq:isolate-n-descent}
f^{(n)}(n\tau)
=
\Phi_n(f(\tau))
-
\sum_{\substack{ad=n\\a<n}}S_{a,d}(\tau).
\end{equation}
Because $N\mid L$, the function $\Phi_n(f)$ is invariant under $\Gamma_0(L)$. Every branch sum on the right is also invariant under $\Gamma_0(L)$ by the preceding argument. Hence
\[
g_n(\tau):=f^{(n)}(n\tau)
\]
is invariant under $\Gamma_0(L)$.

Let
\[
\beta_n=\begin{pmatrix}n&0\\0&1\end{pmatrix}.
\]
Since $g_n=f^{(n)}\circ\beta_n$, conjugation gives invariance of $f^{(n)}$ under
\begin{equation}\label{eq:conjugate-subgroup}
\beta_n\Gamma_0(L)\beta_n^{-1}
=
\Gamma_0(L/n)\cap\Gamma^0(n).
\end{equation}
Here
\[
\frac{L}{n}=\frac{N}{(N,n)}.
\]
If $\gamma=\left(\begin{smallmatrix}A&B\\C&D\end{smallmatrix}\right)$ with $L\mid C$, then
\[
\beta_n\gamma\beta_n^{-1}
=
\begin{pmatrix}A&nB\\C/n&D\end{pmatrix},
\]
whose upper-right entry is divisible by $n$ and whose lower-left entry is divisible by $L/n$. Conversely every matrix in the intersection on the right of \eqref{eq:conjugate-subgroup} arises in this way.

The ordinary integral-power $q$-expansion of $f^{(n)}$ gives invariance under $T:\tau\mapsto\tau+1$. Lemma~\ref{lem:generation}, with $m=L/n=N/(N,n)$, then gives invariance under the full group $\Gamma_0(N/(N,n))$.
\end{proof}

\begin{remark}\label{rem:carnahan-comparison}
Carnahan records \eqref{eq:descent-level} for completely replicable functions with rational integer coefficients, using the exhaustive enumeration of \cite{alexander-cummins-mckay-simons} and the fixing groups computed in \cite{ferenbaugh}; see \cite[Corollary~5.6]{carnahan}. Theorem~\ref{thm:congruence-descent} derives the formula directly from the replication identity, so neither complete replicability nor coefficient integrality is required.
\end{remark}

The level formula is prime-local. If
\[
N=\prod_p p^{e_p},
\qquad
n=\prod_p p^{r_p},
\]
then the guaranteed level of $f^{(n)}$ is
\begin{equation}\label{eq:valuation-level}
\frac{N}{(N,n)}
=
\prod_p p^{\max(e_p-r_p,0)}.
\end{equation}
The guaranteed $\Gamma_0$-level therefore loses the prime-power factors contributed by the replication index.

\begin{cor}\label{cor:coprime-preserve}
Under the hypotheses of Theorem~\ref{thm:congruence-descent}:
\begin{enumerate}
\item if $(n,N)=1$, then $f^{(n)}$ is invariant under $\Gamma_0(N)$;
\item if $N\mid n$, then $f^{(n)}$ is invariant under $\mathrm{SL}_2(\Z)$.
\end{enumerate}
\end{cor}

\begin{remark}
Cummins and Norton note that Norton's earlier work treats replication by an index prime to $N$ in their Hauptmodul setting, and that the argument extends to irrational coefficients and higher genus \cite{cummins-norton,norton-nonmonstrous}. Corollary~\ref{cor:coprime-preserve}(1) is obtained here from the more general descent formula for arbitrary $n$, under the different starting hypotheses of replicability and $\Gamma_0(N)$-invariance.
\end{remark}

\begin{cor}\label{cor:J-final-level}
Let $f$ satisfy the hypotheses of Theorem~\ref{thm:congruence-descent}. If $N\mid n$, then
\begin{equation}\label{eq:multiple-N-J}
f^{(n)}=J.
\end{equation}
In particular,
\[
f^{(N)}=J.
\]
\end{cor}

\begin{proof}
By Corollary~\ref{cor:coprime-preserve}, $f^{(n)}$ is invariant under $\mathrm{SL}_2(\Z)$. Proposition~\ref{prop:analytic-replicates} shows that it is holomorphic on $\Hh$, and by normalization
\[
f^{(n)}(\tau)=q^{-1}+O(q)
\]
at infinity. Therefore $f^{(n)}-J$ descends to a holomorphic function on the compact modular curve $X(1)$. It is constant. Both functions have zero constant term at infinity, so the constant is zero.
\end{proof}

\begin{cor}\label{cor:Hauptmodul-note}
Let $f$ be a normalized replicable Hauptmodul for a genus-zero Fuchsian group $G$ containing $\Gamma_0(N)$, with width one at infinity. Then $f$ is $J$-final and $f^{(N)}=J$.
\end{cor}

\begin{proof}
The normalized Hauptmodul has a simple pole at the cusp at infinity and, as a generator of the function field of the genus-zero quotient, has no other poles. In particular it is holomorphic on $\Hh$. Since $\Gamma_0(N)\subseteq G$, it is $\Gamma_0(N)$-invariant. Apply Corollary~\ref{cor:J-final-level}.
\end{proof}

In Corollary~\ref{cor:Hauptmodul-note}, the genus-zero hypothesis is used only to identify $f$ as a Hauptmodul. Once replicability, holomorphy, and $\Gamma_0(N)$-invariance are given, the proof of $J$-finality does not use genus zero.

For a completely replicable function the descent law is compatible with composition. If
\[
N_m=\frac{N}{(N,m)},
\]
then
\begin{equation}\label{eq:semigroup-level}
\frac{N_m}{(N_m,n)}
=
\frac{N}{(N,mn)}.
\end{equation}
This follows prime by prime from \eqref{eq:valuation-level}. Applying the descent theorem first to $f^{(m)}$ and then to its $n$th replicate therefore gives the same guaranteed congruence level as applying it once to $f^{(mn)}$.

\section{$J$-finality and the terminal replicate}

Let $f$ be completely replicable of finite order $K$. The order relation \eqref{eq:finite-order} singles out one replicate.

\begin{defn}\label{def:terminal}
The terminal replicate of $f$ is
\[
f^{\mathrm{term}}:=f^{(K)}.
\]
\end{defn}

\begin{remark}\label{rem:terminal-independent}
The terminal replicate is independent of the chosen finite replication order. If both $K$ and $K'$ satisfy \eqref{eq:finite-order}, then
\[
f^{(K)}=f^{((K,K'))}=f^{(K')}.
\]
Hence Definition~\ref{def:terminal} is intrinsic to the finite-order replication tower.
\end{remark}

Every branch of the replication tower reaches this order-one replicate after further replication.

\begin{prop}\label{prop:common-terminal}
Let $f$ be completely replicable of finite order $K$. For every $m\geq1$, put $g=(m,K)$. Then
\begin{equation}\label{eq:common-terminal}
\left(f^{(m)}\right)^{(K/g)}=f^{(K)}.
\end{equation}
\end{prop}

\begin{proof}
Complete replicability gives
\[
\left(f^{(m)}\right)^{(K/g)}
=f^{(mK/g)}.
\]
Write $m=gm_0$ and $K=gK_0$ with $(m_0,K_0)=1$. Then
\[
\left(\frac{mK}{g},K\right)
=(m_0K,K)=K.
\]
The finite-order relation \eqref{eq:finite-order} therefore yields
\[
f^{(mK/g)}=f^{(K)}.
\]
\end{proof}

\begin{prop}\label{prop:J-terminal-equivalence}
Let $f$ be completely replicable of finite order $K$. Then the following are equivalent:
\begin{enumerate}
\item $f$ is $J$-final;
\item $f^{(K)}=J$.
\end{enumerate}
\end{prop}

\begin{proof}
The implication (2) to (1) is immediate. Suppose that $f^{(n)}=J$ for some $n$. By finite order,
\[
f^{(n)}=f^{((n,K))}.
\]
Set $d=(n,K)$. Then $d\mid K$ and $f^{(d)}=J$. By complete replicability,
\[
f^{(K)}
=\left(f^{(d)}\right)^{(K/d)}.
\]
Every replicate of $J$ equals $J$, hence $f^{(K)}=J$.
\end{proof}

There can be more than one $J$-final index below $K$, so the replication order need not equal the least $J$-final index. The divisor set below records these indices.

\begin{defn}\label{def:J-divisors}
For a completely replicable function of finite order $K$, define
\[
\mathcal D_J(f)=\{d:d\mid K,\ f^{(d)}=J\}.
\]
\end{defn}

\begin{prop}\label{prop:J-divisor-filter}
The set $\mathcal D_J(f)$ is upward closed in the divisor lattice of $K$: if $d\in\mathcal D_J(f)$ and $d\mid e\mid K$, then $e\in\mathcal D_J(f)$. Moreover,
\begin{equation}\label{eq:J-indices}
f^{(n)}=J
\quad\Longleftrightarrow\quad
(n,K)\in\mathcal D_J(f).
\end{equation}
In particular, $f$ is $J$-final if and only if $K\in\mathcal D_J(f)$.
\end{prop}

\begin{proof}
If $d\mid e$ and $f^{(d)}=J$, complete replicability gives
\[
f^{(e)}=\left(f^{(d)}\right)^{(e/d)}=J.
\]
This proves upward closure. The finite-order identity gives
\[
f^{(n)}=f^{((n,K))},
\]
which is exactly \eqref{eq:J-indices}. The last assertion is Proposition~\ref{prop:J-terminal-equivalence} in this notation.
\end{proof}

Congruence invariance yields an explicit divisor in $\mathcal D_J(f)$.

\begin{cor}\label{cor:congruence-J-divisor}
Let $f$ be completely replicable of finite order $K$, holomorphic on $\Hh$, and invariant under $\Gamma_0(N)$. Then
\begin{equation}\label{eq:gcd-J-divisor}
(N,K)\in\mathcal D_J(f).
\end{equation}
In particular, $f^{(K)}=J$.
\end{cor}

\begin{proof}
Corollary~\ref{cor:J-final-level} gives $f^{(N)}=J$. Since $f^{(N)}=f^{((N,K))}$, the divisor $(N,K)$ lies in $\mathcal D_J(f)$. Upward closure then gives $K\in\mathcal D_J(f)$.
\end{proof}

Martin's modularity theorem gives the complementary implication: $J$-final complete replication forces congruence invariance \cite{martin}. Corollary~\ref{cor:congruence-J-divisor} shows that, once congruence invariance is present, ordinary replication already forces $J$-finality, while complete replication identifies the terminal replicate with $J$.

\section{The terminal alternative}

\begin{prop}\label{prop:trig-family}
For $c\in\C$, put
\[
u_c(\tau)=q^{-1}+cq.
\]
Then $u_c$ is completely replicable with
\begin{equation}\label{eq:trig-replicates}
u_c^{(a)}=u_{c^a}.
\end{equation}
\end{prop}

\begin{proof}
Set $x=q^{-1}+cq$. Define polynomials recursively by
\[
P_0(X)=2,
\qquad
P_1(X)=X,
\qquad
P_n(X)=XP_{n-1}(X)-cP_{n-2}(X).
\]
An induction gives
\[
P_n(q^{-1}+cq)=q^{-n}+c^nq^n.
\]
Hence $P_n$ is the $n$th Faber polynomial of $u_c$.

For a factorization $ad=n$ and $\zeta_d=e^{2\pi i/d}$,
\[
u_{c^a}\!\left(\frac{a\tau+b}{d}\right)
=
\zeta_d^{-b}q^{-a/d}+c^a\zeta_d^bq^{a/d}.
\]
If $d>1$, summing over $b\bmod d$ annihilates both terms. If $d=1$, then $a=n$ and the contribution is
\[
q^{-n}+c^nq^n=P_n(u_c(\tau)).
\]
The replication identity therefore holds with the family \eqref{eq:trig-replicates}. Also,
\[
\left(u_c^{(a)}\right)^{(b)}
=u_{(c^a)^b}
=u_{c^{ab}}
=u_c^{(ab)},
\]
so the replication is complete.
\end{proof}

The functions
\[
u_0=q^{-1},
\qquad
u_1=q^{-1}+q
\]
have replication order one. The function $u_{-1}=q^{-1}-q$ does not, since
\[
u_{-1}^{(2)}=u_1.
\]

\begin{prop}\label{prop:order-one-classification}
Let $h$ be a completely replicable function of replication order one. Then
\[
h\in\{J,u_0,u_1\}.
\]
\end{prop}

\begin{proof}
Proposition~3.3 of Kozlov shows that an order-one completely replicable function satisfies modular equations of every order $n\geq2$ \cite{kozlov}. The classification recorded in \cite[Theorem~2]{gannon-hauptmodul} gives
\[
J,\qquad q^{-1},\qquad q^{-1}+q,\qquad q^{-1}-q
\]
as the only possibilities. The last function has replication order greater than one, since Proposition~\ref{prop:trig-family} gives
\[
u_{-1}^{(2)}=u_1\neq u_{-1}.
\]
\end{proof}

\begin{thm}\label{thm:terminal-trichotomy}
Let $f$ be a completely replicable function of finite order $K$. Then
\begin{equation}\label{eq:terminal-three}
f^{(K)}\in\{J,u_0,u_1\}
=\{J,q^{-1},q^{-1}+q\}.
\end{equation}
\end{thm}

\begin{proof}
Kozlov's order formula gives the replication order
\[
\frac{K}{(K,K)}=1
\]
for $f^{(K)}$ \cite{kozlov}. The result follows from Proposition~\ref{prop:order-one-classification}.
\end{proof}

The trichotomy requires no integrality or algebraicity assumption on the Fourier coefficients. It sharpens when the original function has symmetry beyond translations.

For a holomorphic function $f$ on $\Hh$, write
\[
G(f)=
\left\{
\gamma\in\mathrm{SL}_2(\mathbb R):
f(\gamma\tau)=f(\tau)\text{ for all }\tau\in\Hh
\right\}.
\]
The translation subgroup is
\[
G_{\mathrm{tr}}=\{\pm T^m:m\in\Z\}.
\]

\begin{cor}\label{cor:nontranslation-J-final}
Let $f$ be completely replicable of finite order $K$. If
\[
G(f)\neq G_{\mathrm{tr}},
\]
then
\[
f^{(K)}=J.
\]
In particular, $f$ is $J$-final.
\end{cor}

\begin{proof}
For a finite-order completely replicable function of order $K$, Kozlov's modular-equation result gives modular equations of every order coprime to $K$ \cite{kozlov}; see also \cite[Section~8]{cummins-gannon}. Hence $f$ satisfies modular equations for every $n\equiv1\pmod K$. Since $G(f)$ is larger than the translation group, the nontranslation case of the theorem of Cummins and Gannon applies without an arithmetic hypothesis on the coefficients. It follows that $f$ is a Hauptmodul for $G(f)$ and that $G(f)$ contains $\Gamma_0(N)$ with finite index for some $N\mid K^\infty$ \cite[Theorem~1.3]{cummins-gannon}. Corollary~\ref{cor:congruence-J-divisor} gives
\[
f^{(K)}=J.
\]
\end{proof}

\begin{cor}\label{cor:terminal-translation}
If the terminal replicate of $f$ is $q^{-1}$ or $q^{-1}+q$, then
\[
G(f)=G_{\mathrm{tr}}.
\]
\end{cor}

\begin{proof}
This is the contrapositive of Corollary~\ref{cor:nontranslation-J-final}.
\end{proof}

\begin{cor}\label{cor:terminal-characterization}
Let $f$ be completely replicable of finite order $K$. Then $f$ is $J$-final if and only if
\[
f^{(K)}=J.
\]
The two remaining terminal possibilities, $q^{-1}$ and $q^{-1}+q$, are not $J$-final.
\end{cor}

\begin{proof}
The first assertion is Proposition~\ref{prop:J-terminal-equivalence}. If the terminal replicate is $u_0$ or $u_1$ and some replicate were $J$, Proposition~\ref{prop:J-terminal-equivalence} would force the terminal replicate to be $J$, a contradiction.
\end{proof}

\begin{example}\label{ex:minus}
The function
\[
u_{-1}(\tau)=q^{-1}-q
\]
is completely replicable and
\[
u_{-1}^{(2)}=q^{-1}+q.
\]
Thus $u_{-1}$ has a finite-order replication tower with trigonometric terminal replicate.
\end{example}

\begin{remark}\label{rem:replication-roots}
Theorem~\ref{thm:terminal-trichotomy} concerns only the terminal object. It does not classify the functions above $u_0$ or $u_1$ in a replication tower. The relation $u_{-1}^{(2)}=u_1$ gives a nontrivial replication root of a terminal trigonometric function. The computational classification of Alexander, Cummins, McKay, and Simons assumes uniqueness of the replication roots of $u_0$ and $u_1$ \cite{alexander-cummins-mckay-simons}; no such assumption is needed for Theorem~\ref{thm:terminal-trichotomy}.
\end{remark}

\section{Relation with earlier modularity results}

Theorem~\ref{thm:congruence-descent} complements earlier implications in the subject. Cummins and Norton start from genus-zero modular functions and obtain replication \cite{cummins-norton}, while Martin derives modular invariance from sufficiently strong complete replication together with $J$-finality \cite{martin}. Results of Kozlov and of Cummins and Gannon use modular equations to obtain modularity, genus-zero, or trigonometric alternatives \cite{kozlov,cummins-gannon}.

The present argument begins from congruence invariance. Carnahan records the same level formula for integral completely replicable functions using the classified list and Ferenbaugh's fixing groups \cite[Corollary~5.6]{carnahan}, while Norton's earlier work treats the coprime-index case in the Hauptmodul setting \cite{norton-nonmonstrous,cummins-norton}. Theorem~\ref{thm:congruence-descent} gives the formula for arbitrary indices directly from the replication identity, without complete replicability or arithmetic assumptions on the coefficients. The terminal trichotomy enters only after finite complete replication is imposed, when Kozlov's order-one modular-equation result and the classification recorded by Gannon give the three possibilities used above \cite{kozlov,gannon-hauptmodul}.

\end{document}